%% file: 150428tribranch.tex
\documentclass[10pt]{amsart}

\usepackage{amsmath}
\usepackage{amssymb}
\usepackage{amsthm}
\usepackage[dvips]{graphicx}

\usepackage{tikz}
\usetikzlibrary{matrix,arrows}

\usepackage[utf8]{inputenc}

\usepackage{cite,hyperref}

\DeclareMathOperator{\Hom}{Hom}

\DeclareMathOperator{\rank}{rank}
\DeclareMathOperator{\MT}{M}
\DeclareMathOperator{\Id}{Id}

\newcommand{\IZ}{\mathbb{Z}}
\newcommand{\IR}{\mathbb{R}}

\newcommand{\IC}{\mathbb{C}}

\newcommand{\Map}[3]{#1 \colon #2 \rightarrow #3}

\newcommand{\Sl}[1]{\operatorname{SL}(#1)}
\newcommand{\Book}{\operatorname{B}}

\newcommand{\Comp}[1]{\operatorname{Comp}(#1)}

\theoremstyle{plain}
\newtheorem{thm}{Theorem}[section]

\newtheorem{lem}[thm]{Lemma}

\theoremstyle{definition}
\newtheorem{defn}[thm]{Definition}
\newtheorem{constr}[thm]{Construction}
\newtheorem{exm}[thm]{Example}

\theoremstyle{remark}
\newtheorem{rem}[thm]{Remark}
\newtheorem{que}[thm]{Question}

\begin{document}
\title{A note on the existence of essential tribranched surfaces}
\author{Stefan Friedl}
\address{Department of Mathematics, University of Regensburg, Germany}
\email{sfriedl@gmail.com}
\author{Takahiro Kitayama}
\address{Department of Mathematics, Tokyo Institute of Technology, Japan}
\email{kitayama@math.titech.ac.jp }
\author{Matthias Nagel}
\address{Department of Mathematics, University of Regensburg, Germany}
\email{matthias.nagel@mathematik.uni-regensburg.de}
\urladdr{http://homepages.uni-regensburg.de/~nam23094/}
\begin{abstract}
The second author and Hara introduced the notion of an essential tribranched surface that is a generalisation of the notion of an essential embedded surface in a 3-manifold. We show that any 3-manifold for which the fundamental group has at least rank four admits an essential tribranched surface.
\end{abstract}
\maketitle

\section{Introduction}
Let $M$ be a $3$-manifold. Here and throughout the paper all $3$-mani\-folds
are assumed to be compact, connected and orientable.  A properly embedded,
connected, orientable surface $\Sigma\ne S^2$ in $M$ is called
\emph{essential} if the inclusion induced map $\pi_1(\Sigma)\to \pi_1(N)$ is a
monomorphism and if $\Sigma$ is not boundary-parallel.  A
$3$-manifold is called \emph{Haken} if it is irreducible and if it admits an
essential surface.

Haken mani\-folds play a major role in 3-manifold topology. For example, Waldhausen \cite{Wa68} gave a solution to the homeomorphism problem for Haken 3-mani\-folds, Thurston \cite{Th82} proved the Geometrisation Theorem for Haken 3-mani\-folds and recently Wise \cite{Wi12} showed that  Haken hyperbolic 3-mani\-folds are virtually fibred.
 The principle in each case is that the existence of an essential (in \cite{Wi12} geometrically finite)  surface leads to a hierarchy, which allows a proof by induction. 
 
It is well-known and straightforward to show that any irreducible $3$-manifold $M$ with $b_1(M)\geq 1$ is Haken. 
This implies in particular that any irreducible $3$-manifold $M\ne D^3$ with non-empty boundary is Haken.
A particularly interesting source of essential surfaces in a $3$-manifold $M$ is provided by the work Culler and Shalen \cite{Culler83}. They showed how ideal points of the character variety $\Hom( \pi_1(M), \Sl{2, \IC} ) // \Sl{2, \IC}$ give rise to essential surfaces in  $M$. 

Despite the abundance of Haken mani\-folds,  there are many examples of closed, irreducible $3$-mani\-folds that are not Haken. For example most Dehn surgeries along the Figure-8 knot are hyperbolic but not Haken.
We refer to \cite{Ha82,HT85,Ag03} and also the discussion in \cite[(C.17)]{AFW12} for details and many more examples.
It is still an open, and very interesting question, whether or not a `generic' $3$-manifold is Haken.

Recently  Hara and the second author \cite{Kitayama14} generalised the work of
Culler and Shalen \cite{Culler83} to higher dimensional character varieties.
More precisely, let $M$ be an irreducible $3$-manifold with non-empty boundary and let $n \geq 3$ be a natural number.
It is shown in \cite{Kitayama14}
that any ideal point of the character variety
$\Hom( \pi_1(M), \Sl{n, \IC}) // \Sl{n, \IC}$ gives rise to an essential \emph{tribranched} surface.
For $n = 3$ this even holds if $M$ has no boundary.
We will recall the definition of an essential tribranched surface in Section \ref{section:tribranched}.
For the moment  it suffices to know that an essential surface is also an essential tribranched surface, but as the name already suggests, tribranched surfaces are allowed to have a certain controlled type of branching.

The following question was raised in \cite[Question~6.5]{Kitayama14}

\begin{que}
Does every aspherical 3-manifold contain an essential tribranched surface?
\end{que}

The goal of this paper is to give an affirmative answer  for `most' $3$-mani\-folds.
More precisely, given a finitely generated group $\pi$ we denote by $\rank \pi$ the minimal number of generators of $\pi$. The following is  our main theorem.

\begin{thm}\label{mainthm}
Let $M$ be a closed $3$-manifold with $\rank \pi_1(M) \geq 4$. Then $M$
admits an essential tribranched surface.
\end{thm}

We will construct these tribranched surfaces using  open book decompositions. In particular, our approach is quite different from the approach taken in \cite{Kitayama14}.

\subsection*{Conventions}
All $3$-mani\-folds  are assumed to be compact, connected and orientable, unless it says explicitly otherwise. Furthermore all self-diffeomorphisms of surfaces are assumed to be orientation preserving.
We identify $\IR/\IZ$ with $S^1$ via the map $t\mapsto e^{it}$.

\subsection*{Acknowledgements}
The first and the third author were  supported by the SFB 1085 `Higher Invariants' at the Universit\"at Regensburg funded by the Deutsche Forschungsgemeinschaft (DFG).
The second author was supported by JSPS KAKENHI Grant Number 26800032.

%==============================================
\section{Tribranched surfaces}\label{section:tribranched}
We recall the definition of a tribranched surface from \cite{Kitayama14}.
In this section $M$ will always denote a closed $3$-manifold.
%\todo{Describe previous use of tribranched surfaces and
%differences to Kitayama definition, see Williams - Expanding Attractors 
%and Gabai - Essential Laminations in 3-mani\-folds}

\begin{defn} Let $\Sigma \subset M$ be a subset and $X$ another $3$-manifold.
\begin{enumerate}
\item For a subset $\Lambda$ of $X$ we say that an open set $U$ of $M$ 
is \emph{modelled} on $\Lambda \subset X$ if there exists an open set $V \subset X$ and a 
homeomorphism $\Map \phi U V$ with $\phi( \Sigma \cap U) = \Lambda \cap V$.
\item For a collection $\mathcal{A}$ of subsets of $X$
we say  $\Sigma$ is  \emph{locally modelled} on $\mathcal{A}$ 
if given any point $x \in \Sigma$ there exist an open neighbourhood $U_x$ of $x$ in $M$ and a subset $\Lambda_x \in \mathcal{A}$
such that $U_x$ is modelled on $\Lambda_x \subset X$.
\end{enumerate}
\end{defn}

\begin{exm}
A $2$-dimensional submanifold of a 3-manifold is a subset locally modelled on $\IR^2 \times \{0\}\subset \IR^3$.
\end{exm}

\begin{defn}
\begin{enumerate}
\item A \emph{pre-tribranched surface} $\Sigma$ in $M$ is a closed subset of $M$ locally modelled on the collection
of the following subsets:
\begin{enumerate}
\item $\IR^2 \times \{0\} \subset \IR^3$,
\item $\IR \times \{ r e^{2\pi i k/3} : k=0,1,2 \text{ and } r \in \IR_{\geq 0} \} \subset \IR \times \IC = \IR^3$,
\end{enumerate}
The collection of points which do not have an open neighbourhood 
modelled on $\IR^2\times \{0\} \subset \IR^3$ is called
the \emph{branching set} $C(\Sigma)$. 
Each component of $\Sigma\setminus C(\Sigma)$ is called a \emph{branch}
and each component of $M\setminus \Sigma$ is called a \emph{block}.
\item A \emph{tribranched surface} is a pre-tribranched surface such that
\begin{enumerate}
\item each branch is an orientable surface and
\item each component $C$ of the branching set $C(\Sigma)$ has an open neighbourhood $U$
modelled on 
\[ S^1 \times \{ r e^{2\pi i k/3} : k=0,1,2 \text{ and } r \in \IR_{\geq 0} \} \subset S^1 \times \IC.\]
\end{enumerate}
\end{enumerate}
\end{defn}

\begin{rem}
For a pre-tribranched surface $\Sigma$ the subset $C(\Sigma)$ is 
a $1$-dimensional submanifold of $M$. Each branch is a 
surface and each block is a 3-manifold.
\end{rem}

Later on we will also use the following definitions.

\begin{defn} Let $\Sigma$ be a tribranched surface in $M$. 
\begin{enumerate}
\item We say that a \emph{branch $S$ of $\Sigma$ lies in the closure of a block $N$} if $S$ is contained in the closure of $N$ in $M$. 
\item The \emph{compactification $S^c$ of a branch  $S$} is the abstract compact surface  whose interior is $S$.
\item The \emph{compactification $N^c$ of a block $N$} is the abstract compact 3-manifold whose interior is $N$.
\end{enumerate}
\end{defn}

Note that given a branch $S$ there exists a canonical map $S^c\to M$.
In general this  canonical map $S^c\to M$ might identify components of the boundary of $S^c$. Similarly, given a block $N$ there exists a canonical map $N^c\to M$.
Also,  if a branch $S$ lies in the closure of a block $N$, then there exists a canonical injective map $S\to N^c$

\begin{defn}\label{Def:essential}
A tribranched surface $\Sigma$ is called \emph{essential} if 
all conditions below are met.
\begin{enumerate}
\item\label{ESnoDisc} No branch of $\Sigma$ is a disc.
\item\label{ESnoBall} No component of $\Sigma$ is contained in a ball.
\item\label{ESPerchInj} For each block $N$ and each
branch $S$ in the closure of $N$ the associated homomorphism 
$\pi_1(S) \rightarrow \pi_1(N^c)$ is injective.
\item\label{ESnoSurj} For each block $N$ the inclusion $N\subset M$ does not
induce a surjection on the fundamental group.
\end{enumerate}
\end{defn}
%\todo{What do to with ETBS4? ++++++++=[SF] I would ignore this}

It follows easily from the definitions that an essential surface is also an essential tribranched surface.

\begin{rem}
In \cite{Kitayama14} the notion of an essential tribranched surface is introduced for any compact 3-manifold, not necessarily closed. The definition is verbatim the same as above, except that one also demands that the extra condition
\begin{enumerate}
\item[(5)]\label{ESnotBdy} No component of $\Sigma$ is boundary parallel.
\end{enumerate}
holds.
In the remainder of this paper we will only be concerned with closed 3-manifolds.
\end{rem}

%====================================================
\section{Open books}
In this section we will recall some basic definitions and facts about open books in 3-manifolds.

\begin{constr}
Let $F$ be an orientable, compact, connected surface  and let $\phi\colon F \rightarrow F$ be an orientation preserving homeomorphism
that restricts to the identity in a neighbourhood of the boundary $\partial F$.
We form the mapping torus $\MT(F, \phi)$. This is the quotient 
\begin{align*}
\begin{array}{rcl}
\hspace{3.5cm} \MT(F, \phi) &:= &\IR \times F /{(t, \phi(x)) \sim (t+1, x)}.
\intertext{We have natural identifications  }
\partial \MT(F, \phi)& =& \MT(\partial F, \Id) = S^1\times \partial F ,
\end{array}
\end{align*}
which show that  the boundary $\partial \MT(F, \phi)$ is a disjoint union of  tori, where each boundary torus comes with a preferred homeomorphism to the standard torus.
Denote the collection of boundary components of $F$ by $\Comp{\partial F}$.
Into these we glue copies of $D^2\times S^1$, using the above preferred homeomorphisms of the boundary components of $\MT(F,\phi)$ with the standard torus $S^1\times S^1=\partial(D^2\times S^1)$.
We obtain a $3$-manifold without boundary
\begin{align*}
\Book(F, \phi) &:= \MT(F, \phi) \cup_{S^1\times \partial F } \bigcup_{b\in \Comp{\partial F}} (D^2 \times S^1)_b.
\end{align*}
Evidently, this construction is independent of the choice of orientation preserving identification 
\begin{align*}
\partial F \cong \bigcup_{b\in \Comp{\partial F}} (S^1)_b.
\end{align*}
For each boundary component $b \in \Comp{\partial F}$ we obtain an embedded curve
$\{0\} \times S^1  \subset (D^2 \times S^1)_b$. The union of these curves  is called the \emph{spine} of the open book.
We call the resulting
manifold $\Book(F, \phi)$ the \emph{open book}  with \emph{page} $F$ and
\emph{monodromy} $\phi$.
\end{constr}

\begin{defn}
An \emph{open book decomposition} of a $3$-manifold $M$ is a homeomorphism
to a manifold $\Book(F, \phi)$ as constructed above.
\end{defn}

One of the first non-trivial results in 3-manifold topology was Alexander's proof in 1920 \cite{Al20} that every closed $3$-manifold admits an open book decomposition. We refer to \cite{Et06} for a more modern account.

\begin{thm}[Alexander]
Every closed $3$-manifold admits an open book decomposition.
\end{thm}

Given an open book decomposition of a $3$-manifold one can obtain a new one whose page has a smaller Euler characteristic using stabilisation, see e.g.\ \cite[Corollary~2.21]{Et06}. We thus obtain the following slightly stronger statement.

\begin{thm}\label{thm:openbookexists}
Given any closed $3$-manifold $M$ and $k\in \Bbb{Z}$ there exists  an open book decomposition of $F$ such that the page $F$ satisfies $\chi(F)\leq k$. 
\end{thm}

%====================================================
\section{From open books to tribranched surfaces}

In the following we construct a tribranched surface in an open book.
\begin{constr}[Naive tribranched surface]
Let $F$ be a compact, orientable, connected surface  and let $\phi\colon F \rightarrow F$ be an orientation preserving homeomorphism
that restricts to the identity in a neighbourhood of the boundary $\partial F$.
The corresponding  open book $\Book(F, \phi)$ is given
by a gluing
\begin{align*}
\Book(F, \phi) = \left( \IR \times F / \sim \right) \cup_{S^1\times \partial F } 
	\bigcup_{b\in \Comp{\partial F}} (D^2 \times S^1)_b.
\end{align*}
We construct subsets $\Sigma_{I,b}$ in each copy $(D^2 \times S^1)_b$ as follows
\begin{align*}
\Sigma_{I, b} := \{ r e^{2\pi i k/3} : k=0,1,2 \text{ and } 0 \leq r \leq 1 \}\times S^1\,\, \subset (D^2 \times S^1)_b.
\end{align*}
These subsets can be completed to a tribranched
surface $\Sigma$ in $\Book(F, \phi)$ by setting
\begin{align*}
\Sigma\, :=\, \{0, \tfrac{1}{3}, \tfrac{2}{3}\} \times F\,\, \cup\,\, \bigcup_{b\in \Comp{\partial F}} \Sigma_{I,b}.
\end{align*}
\end{constr}

\begin{lem}\label{lem:naive}
If  $\chi(F)\leq 0$, then the tribranched surfaces $\Sigma \subset \Book(F, \phi)$ constructed above
fulfil conditions $(\ref{ESnoDisc})$ and $(\ref{ESPerchInj})$ stated in Definition \ref{Def:essential}. 
\end{lem}

\begin{proof}
The branching set is the spine of the open book.
All branches are diffeomorph to the interior $\operatorname{int}(F)$ of $F$ and therefore by our assumption on $\chi(F)$ none of the branches is a disc.
Each block is homeomorphic to $(0,1) \times \operatorname{int}(F)$. If $S$ is a branch in the closure of a block $N$, then the inclusion $S\to N^c$   corresponds
to the inclusion $\{0\}\times \operatorname{Int}(F)  \subset [0,1] \times F$. As a result we obtain (3).
\end{proof}

The combination of Theorem \ref{thm:openbookexists} and Lemma \ref{lem:naive} shows that any closed $3$-manifold admits a tribranched surface that satisfies 
conditions $(\ref{ESnoDisc})$ and $(\ref{ESPerchInj})$ of Definition \ref{Def:essential}. 
On the other hand, not every closed $3$-manifold admits an essential tribranched surface.
For example it is straightforward to see that no spherical $3$-manifold can admit a tribranched surface that satisfies conditions  $(\ref{ESnoDisc})$,  $(\ref{ESnoBall})$ and $(\ref{ESnoSurj})$.
\smallskip

In the following we will modify the above construction to obtain an essential tribranched surface 
in any open book $\Book(F, \phi)$ under the additional hypothesis that the rank of the fundamental group is at least four.
This is done by chopping down the blocks of the above tribranched surface
so that it becomes impossible for them to surject on the fundamental group.
We first collect some facts on pants decomposition of surfaces.

\begin{defn}
\begin{enumerate}
\item A collection $\mathcal{C}$ of disjointly embedded curves into 
a surface $F$ is called a \emph{pants decomposition} if $F$ cut along $\mathcal{C}$ 
is a disjoint union of thrice punctured spheres.
\item A tuple $(\mathcal{C}_0, \ldots, \mathcal{C}_n)$ of pants decompositions
is a \emph{path}
connecting $\mathcal{C}_0$ and $\mathcal{C}_n$ if for each $0 \leq k < n$
there is a curve $\gamma \in \mathcal{C}_k$ and a curve 
$\gamma' \in \mathcal{C}_{k+1}$
such that $\mathcal{C}_k \setminus \gamma$ is isotopic
 to $\mathcal{C}_{k+1} \setminus \gamma'$.
\end{enumerate}
\end{defn}

It is well-known that any compact, orientable, connected surface $F$ with negative Euler characteristic admits a pants decomposition. 
The following theorem is proved by Hatcher--Lochak--Schneps
\cite[Section~2]{HLS00}, building on earlier work of Hatcher--Thurston
\cite[Appendix]{Hatcher80}.

\begin{thm}
Let $F$ be a compact, orientable, connected surface with  $\chi(F) <0$.
Suppose $\mathcal{C}$ and $\mathcal{C}'$ are pants decompositions.
Then there exists a path connecting $\mathcal{C}$ to $\mathcal{C}'$.
\end{thm}

\begin{rem}
Let $\Book(F, \phi)$ be an open book.
If $\mathcal{C}$ is a pants decomposition of $F$,
then the following is
another pants decomposition:
\begin{align*}
\phi(\mathcal{C}) := \{\phi(\gamma) : \gamma \in \mathcal{C} \}.
\end{align*}
The theorem above guarantees us that there is a path between 
$\mathcal{C}$ and $\phi(\mathcal{C})$.
\end{rem}

First we construct the essential tribranched surface away from the spine and then 
describe how to extend it in a neighbourhood of the spine.
\begin{constr}[Outer tribranched surface]
Let $\Book(F, \phi)$ be an open book with $\chi(F)<0$. 
Choose a pants decomposition $\mathcal{C}$ of $F$.
Let $(\mathcal{C}_0, \ldots, \mathcal{C}_n)$ be a path connecting $\mathcal{C}$
with $\phi(\mathcal{C})$.
We start off with the set 
\begin{align*}
F_n := \{\tfrac{0}{n}, \tfrac{1}{n} \ldots, \tfrac{n-1}{n}\} \times F \,\subset\, \IR \times F / {(t, \phi(x)) \sim (t+1, x)}.
\end{align*}
For $0 \leq k < n$ 
denote $\mathcal{D}_k := \mathcal{C}_k \cap \mathcal{C}_{k+1}$. For each curve $\gamma \in \bigcup_{k}\mathcal{D}_k$  we 
pick a push-off $\gamma^+$ on $F$ and an isotopy $\Map {m_\gamma} {S^1\times [0,1]} \Sigma$
with
\begin{enumerate}
\item $m_\gamma(-,0) = \gamma$ and $m_\gamma(-,1) = \gamma^+$
\item $m_\gamma(-,t)$ is independent of $t$ in a neighbourhood
of the endpoints $0$ and $1$.
\end{enumerate}
For each $k$ we can and will pick the  $m_\gamma$ such that the images are disjoint.
Now we chop the complement of $F_n$ into smaller pieces by adding the following
horizontal pieces:
\begin{align*}
H_{k,\gamma} := \left\{ \left(\tfrac{k+t}{n}, m_\gamma(t) \right) : t \in [0,1] \right\} \subset \IR \times F / \sim.
\end{align*}
for $0 \leq k < n$ and $\gamma \in \mathcal{D}_k$.
The \emph{outer tribranched surface} is the subset
\begin{align*}
\Sigma_O := F_n \cup \bigcup_{ \substack{0 \leq k < n,\\ \gamma \in \mathcal{D}_k}} H_{k,\gamma}.
\end{align*}
\end{constr}

\begin{figure}[h]
\begin{center}
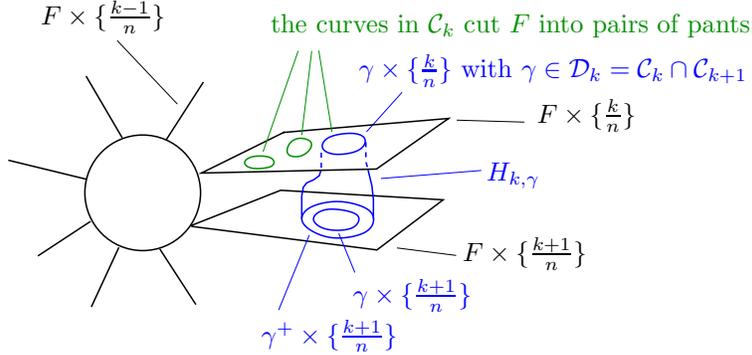
\caption{Schematic picture of the outer tribranched surface.} \label{fig:outer}
\end{center}
\end{figure}

%\begin{constr}[Inner tribranched surface]
%We denote by $\overline{xy}$ the straight line connecting two points in $D^2$. Abbreviate
%the points $\frac{1}{2}\cdot  e^{2\pi i k/n} \in D^2$ to $x_k$ for $0 \leq k < n$. The \emph{inner tribranched surface} is
%\begin{align*}
%\Sigma_I := \left( \bigcup_{0 \leq k < n} \overline {x_k x_{k+1} } \cup \bigcup_{0 \leq k < n} \overline {x_k e^{2\pi i k/n} } \right) \times S^1 \subset D^2 \times S^1.
%\end{align*}
%Again we denote the corresponding copy in $(D^2 \times S^1)_b$ by $\Sigma_{I,b}$ for
%$b \in \Comp{\partial F}$.
%\end{constr}

%\begin{figure}[h]
%\begin{center}
%\input{inner-tribranched-surface.pstex_t}
%\caption{Schematic picture of the inner tribranched surface.} \label{fig:inner}
%\end{center}
%\end{figure}

\begin{thm}\label{thm:essentialtribranched}
Let $\Book(F, \phi)$ be an open book with $\chi(F)<0$ such that its fundamental group
is at least of $\rank \pi_1(\Book(F, \phi)) \geq 4$.
Let $\Sigma_O$  be the subset constructed above.
Then the subset 
\begin{align*}
\begin{array}{rcl}
\Sigma& :=& \Sigma_O \quad \cup\quad  \bigcup_{b \in \Comp{\partial F}} (S^1\times S^1)_b
\intertext{is an essential tribranched surface in}
 \hspace{3cm} \Book(F, \phi)&=&\IR \times F / \sim\quad \cup \quad \bigcup_{b \in \Comp{\partial F}} (D^2 \times S^1)_b.
\end{array}\end{align*}
\end{thm}

\begin{proof}
We first show that $\Sigma$ is a tribranched surface. 
A pre-tribranched surface $\Sigma$ in $M$ is a closed subset of $M$ locally modelled on the collection
of the following subsets:
\begin{enumerate}
\item[(a)] $\IR^2 \times \{0\} \subset \IR^3$,
\item[(b)] $\IR \times \{ r e^{2\pi i k/3} : k=0,1,2 \text{ and } r \in \IR_{\geq 0} \} \subset \IR \times \IC = \IR^3$,
\item[(c)]\label{tmodel} $\{0\} \times \IR^2 \cup \left(\IR_{\geq 0} \times \IR \times \{0\}\right) \subset \IR^3$.
\end{enumerate}
Here we  added (c) to the original definition. But it is clear that (c) is just a special case of (b).
Furthermore, a \emph{tribranched surface} is a pre-tribranched surface such that
\begin{enumerate}
\item each branch is an orientable surface and
\item each component $C$ of the branching set $C(\Sigma)$ has an open neighbourhood $U$
modelled on either
\begin{enumerate}
\item $S^1 \times \{ r e^{2\pi i k/3} : k=0,1,2 \text{ and } r \in \IR_{\geq 0} \} \subset S^1 \times \IC$ or
\item \label{ctmodel}$\{0\} \times S^1 \times \IR \cup \left(\IR_{\geq 0}\times S^1 \times \{0\}\right) \subset \IR \times S^1 \times \IR$.
\end{enumerate}
\end{enumerate}
Here we added (b), but again it is clear that (b) is just a special case of (a).
Using this definition of a tribranched surface it is now straightforward to show that $\Sigma$ is indeed a tribranched surface.

It thus remains to show that $\Sigma$ is an essential tribranched surface.
In the following we say that  a subsurface $S$ of $F$ is $\pi_1$-injective, if $\pi_1(S)\to \pi_1(F)$ is a monomorphism.  It is clear from the construction that each branch of $\Sigma$ is one of the following four types:
\begin{enumerate}
\item a component of some $H_{k,\gamma}$, i.e.\ an annulus,
\item $\{\frac{k}{n}\}\times A$ where $A$ is a $\pi_1$-injective annulus on $F$ cobounding some $\gamma$ and the corresponding push-off $\gamma^+$,
\item $\{\frac{k}{n}\}\times S$ where $S$ is a $\pi_1$-injective  thrice punctured sphere on $F$, 
\item $\{\frac{k}{n}\}\times T$ where $T$ is a $\pi_1$-injective subsurface of $F$ that is  the result of gluing two thrice punctured spheres on $F$ along the one component of $\mathcal{C}_k \setminus \mathcal{C}_{k+1}$. 
\end{enumerate}
Note that the last type of branch is a four times punctured sphere. 
It  follows from this list of branches that  (\ref{ESnoDisc}) is satisfied.  Since the surface $\Sigma$ is connected, it follows easily that (\ref{ESnoBall}) is satisfied. 

It is straightforward to see there are precisely two types of blocks:
\begin{enumerate}
\item blocks  of the form $ (\frac{k}{n},\frac{k+1}{n})\times S$, where $S$ is a $\pi_1$-injective subsurface of $F$ that is either a  thrice punctured sphere, or a four-times punctured sphere,
\item blocks that are the interiors of the solid tori $(D^2\times S^1)_b$, $b\in \Comp{\partial F}$.
\end{enumerate}
Together with the above description of branches we easily see that  $\Sigma$ satisfies (\ref{ESPerchInj}).

By the above description of blocks, the fundamental group of each block is generated by at most three elements.
It follows from our hypothesis  $\rank \pi_1(\Book(F, \phi)) \geq 4$ that (\ref{ESnoSurj}) is also satisfied.
Summarising we showed that $\Sigma$ is an essential tribranched surface.
\end{proof}

Now we are finally in a position to prove Theorem~\ref{mainthm}.

\begin{proof}[Proof of Theorem~\ref{mainthm}]
Let $M$ be a closed $3$-manifold with $\rank (\pi_1(M)) \geq 4$. 
By Theorem~\ref{thm:openbookexists} the manifold $M$ admits an open book decomposition $\Book(F, \phi)$  with $\chi(F)<0$.
Thus Theorem~\ref{thm:essentialtribranched}  gives us the desired essential tribranched surface.
\end{proof}

\bibliography{biblio}{}
\bibliographystyle{alpha}
\end{document}

%% file: outer-tribranched-surface.pstex_t
\begin{picture}(0,0)%
\includegraphics{outer-tribranched-surface.eps}%
\end{picture}%
%
%  Created by WinFIG version 4.9 
%  METADATA <version>1.0</version> 
%
\setlength{\unitlength}{829sp}%
\begingroup\makeatletter\ifx\SetFigFont\undefined%
\gdef\SetFigFont#1#2#3#4#5{%
  \reset@font\fontsize{#1}{#2pt}%
  \fontfamily{#3}\fontseries{#4}\fontshape{#5}%
  \selectfont}%
\fi\endgroup%
\begin{picture}(22911,10450)(4157,-10786)
%  METADATA <id>15</id> 
\put(17791,-8071){\makebox(0,0)[lb]{\smash{{\SetFigFont{10}{12.0}{\rmdefault}{\mddefault}{\updefault}{\color[rgb]{0,0,0}$F\times \{\frac{k+1}{n}\}$}%
}}}}
%  METADATA <id>13</id> 
\put(20011,-3931){\makebox(0,0)[lb]{\smash{{\SetFigFont{10}{12.0}{\rmdefault}{\mddefault}{\updefault}{\color[rgb]{0,0,0}$F\times \{\frac{k}{n}\}$}%
}}}}
%  METADATA <id>16</id> 
\put(5176,-991){\makebox(0,0)[lb]{\smash{{\SetFigFont{10}{12.0}{\rmdefault}{\mddefault}{\updefault}{\color[rgb]{0,0,0}$F\times \{\frac{k-1}{n}\}$}%
}}}}
%  METADATA <id>22</id> 
\put(14686,-2581){\makebox(0,0)[lb]{\smash{{\SetFigFont{10}{12.0}{\rmdefault}{\mddefault}{\updefault}{\color[rgb]{0,0,1}$\gamma\times\{\frac{k}{n}\}$ with $\gamma\in\mathcal{D}_k= \mathcal{C}_k\cap \mathcal{C}_{k+1}$}%
}}}}
%  METADATA <id>29</id> 
\put(11761,-10531){\makebox(0,0)[lb]{\smash{{\SetFigFont{10}{12.0}{\rmdefault}{\mddefault}{\updefault}{\color[rgb]{0,0,1}$\gamma^+\times\{\frac{k+1}{n}\}$}%
}}}}
%  METADATA <id>36</id> 
\put(14506,-9316){\makebox(0,0)[lb]{\smash{{\SetFigFont{10}{12.0}{\rmdefault}{\mddefault}{\updefault}{\color[rgb]{0,0,1}$\gamma\times\{\frac{k+1}{n}\}$}%
}}}}
%  METADATA <id>28</id> 
\put(18511,-5656){\makebox(0,0)[lb]{\smash{{\SetFigFont{10}{12.0}{\rmdefault}{\mddefault}{\updefault}{\color[rgb]{0,0,1}$H_{k,\gamma}$}%
}}}}
%  METADATA <id>43</id> 
\put(12046,-1231){\makebox(0,0)[lb]{\smash{{\SetFigFont{10}{12.0}{\rmdefault}{\mddefault}{\updefault}{\color[rgb]{0,.56,0}the curves in $\mathcal{C}_k$ cut $F$ into pairs of pants}%
}}}}
\end{picture}%